\documentclass[11pt,a4paper]{article}

\usepackage{color}
\usepackage{graphics,epsfig}
\usepackage{amsfonts,amssymb,amsmath,amsthm}

\newtheorem{lem}{Lemma}[section]
\newtheorem{cor}[lem]{Corollary}
\newtheorem{thm}[lem]{Theorem}

\theoremstyle{definition}
\newtheorem{defi}[lem]{Definition}
\theoremstyle{remark}
\newtheorem{rem}[lem]{Remark}

\numberwithin{equation}{section}

\newcommand{\ep}{\varepsilon}
\newcommand{\ue}{u^\ep}

\newcommand{\n}{\nabla }
\newcommand{\om}{\Omega}

\newcommand{\emutep}{e^{\mu t^\ep/\ep }}

\newcommand{\ombar}{\overline{\Omega}}

\newcommand{\eun}{\displaystyle{\frac{1}{\ep}}}
\newcommand{\R}{\mathbb{R}}

\newcommand{\vsp}{\vspace{8pt}}
\newcommand{\di}{\displaystyle}
\newcommand{\Pe}{(P^{\;\!\ep})}
\newcommand{\Pz}{(P^{\;\!0})}
\newcommand {\Q}{Q_T}
\newcommand{\regionunzero}{\om _0 ^{(1)}}
\newcommand{\regionzerozero}{\om _0 ^{(0)}}
\newcommand{\regionun}{\om _t ^{(1)}}
\newcommand{\regionzero}{\om _t ^{(0)}}
\newcommand{\support}{\om _t ^{\star}}
\newcommand{\gammasupport}{\Gamma _t ^{\star}}
\newcommand{\nusupport}{\nu _t ^{\star}}

\newcommand{\C}{\mathcal C}
\newcommand{\epai}{\alpha _0}

\title{Interface dynamics of the porous medium equation with a bistable reaction term}
\author{ }
\date{}

\begin{document}

\maketitle \vspace{-20 mm}

\begin{center}

{\large\bf Matthieu Alfaro\footnote{The first author is supported
by the French "Agence Nationale de la Recherche" within the
project IDEE
(ANR-2010-0112-01).} }\\[1ex]
I3M, Universit\'e de Montpellier 2,\\
CC051, Place Eug\`ene Bataillon, 34095 Montpellier Cedex 5, France,\\
[2ex]

{\large\bf Danielle Hilhorst }\\[1ex]
CNRS and Laboratoire de Math\'ematiques,\\
Universit\'e de Paris-Sud 11, 91405 Orsay Cedex, France. \\
[2ex]

\end{center}

\vspace{15pt}


\begin{abstract}

We consider a degenerate partial differential equation arising in
population dynamics, namely the porous medium equation with a
bistable reaction term. We study its asymptotic behavior as a
small parameter, related to the thickness of a diffuse interface,
tends to zero. We prove the rapid formation of transition layers
which then propagate. We prove the convergence to a sharp
interface limit whose normal velocity, at each point, is that of
the underlying
degenerate travelling wave. \\

\end{abstract}

\section{Introduction}\label{intro-poreux}

In this paper we consider the rescaled porous medium equation with
a bistable reaction term
\[
 \Pe \quad\begin{cases}
 u_t=\ep \Delta (u^m)+\eun f(u)&\text{in }\om \times (0,\infty)\vspace{3pt}\\
 \di \frac{\partial (u^m)}{\partial \nu} = 0 &\text{on }\partial \om \times (0,\infty)\vspace{3pt}\\
 u(x,0)=u_0(x) &\text{in }\om\,,
 \end{cases}
\]
and study the sharp interface limit as $\ep\to 0$. Here  $\om$ is
a smooth bounded domain in $\R^N$ ($N\geq 2$), $\nu$ is the
Euclidian unit normal vector exterior to $\partial \om$ and $m
>1$.

We assume that $f$ is smooth, has exactly three zeros $0<a<1$ such
that
\begin{equation}\label{der-f-poreux}
f'(0)<0\,, \qquad f'(a)>0\,, \qquad f'(1)<0\,,
\end{equation}
and that
\begin{equation}\label{positive-speed}
\int _ {0} ^ {1} m u^{m-1} f(u)\, du >0\,.
\end{equation}
The above assumption implies that the speed of the underlying
degenerate travelling wave is positive (see subsection
\ref{ss:materials}), so that the region enclosed by the limit
interface is expanding (see below). This explains why the
requirement \eqref{positive-speed} is convenient for the study of
invasion processes.

As far as the initial data is concerned, we assume that $0\leq u_0
\leq M$ (for some $M>a$) is a $C^2(\ombar)$ function with compact
support
$$ Supp\, u_0:={\rm Cl}\{x\in\om:\, u_0(x)>0\} \subset\subset \om\,. $$
Furthermore we define the {\it initial interface} $\Gamma _0$ by
$$
\Gamma _0:=\{x\in\om:\, u_0(x)=a \}\,,
$$
and suppose that $\Gamma _0$ is a smooth hypersurface without
boundary, such that, $n$ being the Euclidian unit normal vector
exterior to $\Gamma _0$,
\begin{equation}\label{dalltint-poreux}
\Gamma _0 \subset\subset \Omega \quad \mbox { and } \quad \n
u_0(x) \neq 0\quad\text{if $x\in\Gamma _0$\,,}
\end{equation}
\begin{equation}\label{initial-data-poreux}
u_0>a \quad \text { in } \quad \regionunzero\,,\quad u_0<a \quad
\text { in } \quad \regionzerozero\,,
\end{equation}
where $\regionunzero$ denotes the region enclosed by $\Gamma _0$
and $\regionzerozero$ the region enclosed between $\partial \om$
and $\Gamma _0$.

\vskip 8pt Problem $\Pe$ possesses a unique weak solution $\ue$ as
it is explained in  Section \ref{s:preliminaries}. As $\ep
\rightarrow 0$, by formally neglecting the diffusion term, we see
that, in the very early stage, the value of $\ue$ quickly becomes
close to either $1$ or $0$ in most part of $\Omega$, creating a
steep interface (transition layers) between the regions
$\{\ue\approx 1\}$ and $\{\ue\approx 0\}$. Once such an interface
develops, the diffusion term is large near the interface and comes
to balance with the reaction term. As a result, the interface
ceases rapid development and starts to propagate in a slower time
scale. Therefore the limit solution $\tilde u (x,t)$ will be a
step function taking the value $1$ on one side of the moving
interface, and $0$ on the other side.

We shall prove that this sharp interface limit, which we denote by
$\Gamma_t$, obeys the law of motion
\[
 \Pz\quad\begin{cases}
 \, V_{n}=c^*
 \quad \text { on } \Gamma_t \vspace{3pt}\\
 \, \Gamma_t\big|_{t=0}=\Gamma_0\,,
\end{cases}
\]
where $V_n$ is the normal velocity of $\Gamma _t$ in the exterior
direction, and $c^*$ the positive speed of the underlying
travelling wave (see subsection \ref{ss:materials}). Problem $\Pz$
possesses a unique smooth solution on $[0,T^{max})$ for some
$T^{max}>0$. We denote this solution by $\Gamma =\cup _{0\leq t <
T^{max}} (\Gamma_t\times\{t\})$. From now on, we fix $0<T<T^{max}$
and work on $[0,T]$.

We set
\[
Q_T:=\om \times (0,T)\,,
\]
and, for each $t\in [0,T]$, we denote by $\regionun$ the region
enclosed by the hypersurface $\Gamma_t$, and by $\regionzero$ the
region enclosed between $\partial \om$ and $\Gamma_t$. We define a
step function $\tilde u(x,t)$ by
\begin{equation}\label{u}
\tilde u(x,t):=\begin{cases}
\, 1 &\text{in } \regionun\vspace{3pt}\\
\, 0 &\text{in } \regionzero
\end{cases} \quad\text{for } t\in[0,T]\,,
\end{equation}
which represents the formal asymptotic limit of $\ue$ as $\ep\to
0$.

\vskip 8pt Our main result, Theorem \ref{width}, describes both
the emergence and the propagation of the layers. First, it gives
the profile of the solution after a very short initial period: the
solution $\ue$ quickly becomes close to $1$ or $0$, except in a
small neighborhood of the initial interface $\Gamma _0$, creating
a steep transition layer around $\Gamma _0$ ({\it generation of
interface}). The time needed to develop such a transition layer,
which we will denote by $t ^\ep$, is $\mathcal O (\ep|\ln\ep|)$.
Then the theorem states that the solution $\ue$ remains close to
the step function $\tilde u$ on the time interval $[t^\ep,T]$
({\it motion of interface}).

\begin{thm}[Generation and motion of interface]\label{width}
Assume $m\geq 2$. Define $\mu$ as the derivative of $f(u)$ at the
unstable equilibrium $u=a$, that is
$$
\mu=f'(a)\,.
$$

Let $\eta \in(0,\min(a,1-a))$ be arbitrary. Fix $\epai >0$
arbitrarily small. Then there exist positive constants $\ep _0 $
and $\C$ such that, for all $\,\ep \in (0,\ep _0)$ and for all
$(x,t)$ such that $\,t^\ep \leq t \leq T$, where
\begin{equation}\label{tep}
t^\ep:=\mu ^{-1}  \ep |\ln \ep|\,,
\end{equation}
we have
\begin{equation}\label{resultat}
\ue(x,t) \in
\begin{cases}
\,[1-2\ep,1+\eta]&\quad\text{if}\quad x\in\om
_t^{(1)}\setminus\mathcal N_{\C\ep|\ln\ep|}(\Gamma _t)\vsp\\
\,[0,1+\eta]&\quad\text{if}\quad
x\in\Omega\vsp\\
\,[0,\eta]&\quad\text{if}\quad x\in\om_t^{(0)}\setminus\mathcal
N_{\epai}(\Gamma _t)\,,
\end{cases}
\end{equation}
where $\mathcal N _r(\Gamma _t):=\{x\in \om:\,  dist (x,\Gamma
_t)<r\}$ denotes the $r$-tubular neighborhood of $\Gamma _t$.
\end{thm}

\begin{rem}[About the thickness of the interface]Since the
construction of super-solutions is much more involved than that of
sub-solutions, the statement \eqref{resultat} is more accurate in
$\om _ t^{(1)}$ than in $\om _ t^{(0)}$. More precisely, on the
one hand, \eqref{resultat} shows that the convergence to 1 is
uniform \lq\lq inside the interface" except in $\mathcal O
(\ep|\ln\ep|)$ tubular neighborhoods of the sharp interface limit;
on the other hand, \eqref{resultat} only shows that the
convergence to 0 is uniform \lq\lq outside the interface" except
in $\mathcal O (1)$ tubular neighborhoods of the sharp interface
limit.
\end{rem}

\begin{rem}[About the assumption $m\geq 2$]Note that the sub- and super-solutions we shall construct to
study the motion of interface allow $m>1$. Nevertheless, since we
consider {\it not well-prepared initial data}, we need to quote a
generation of interface result from \cite{A-Hil} which is valid
only for $m\geq 2$ (if $1<m<2$ the partial differential equation
is not only degenerate but also singular). When initial data have
a \lq\lq suitable shape", the restriction $m\geq 2$ can be
removed.
\end{rem}

As a direct consequence of Theorem \ref{width}, we have the
following convergence result.

\begin{cor}[Convergence]\label{total}
Assume $m\geq 2$. As $\ep\to 0$, the solution $\ue$ converges to
$\tilde u$ in $\cup _{0<t\leq  T}(\om^{(i)}_t\times\{ t\})$, where
$i=0, 1$.
\end{cor}

\vskip 8 pt For the relevance of nonlinear diffusion in population
dynamics models, we refer the reader to Gurney and Nisbet
\cite{GN}, Gurtin and Mac Camy \cite{GM}: density dependent
diffusion is efficient to study the dynamics of a population which
regulates its size below the carrying capacity set by the supply
of nutrients. Since density dependent equations degenerate at
points where $u=0$, a loss of regularity of solutions occurs and
their support propagates at finite speed.

Let us mention some earlier works on problems involving nonlinear
diffusion that are related to ours. Feireisl \cite{F} has studied
the singular limit of $\Pe$ in the whole space $\R ^N$, which
allows to reduce the issue to the radially symmetric case.
Hilhorst, Kersner, Logak and Mimura \cite{HKLM} have investigated
the singular limit of the equation posed in a bounded domain of
$\R ^N$, with $f(u)$ of the Fisher-KPP type. Note that the authors
in \cite{HKLM} assume the convexity of $\om _0 ^{(1)}$ which
allows them to construct a single super-solution for both the
generation and the motion of interface. Here, we dot not make such
a geometric assumption.

\vskip 8 pt The organization of this work is as follows. In
Section \ref{s:preliminaries}, we briefly recall known results
concerning the well-posedness of Problem $\Pe$. Section
\ref{s:motion} is the body of the paper: we construct sub- and
super-solutions to study the motion of the transition layers.
Finally, we prove Theorem \ref{width} in Section \ref{s:proof}.

\section{Comparison principle,
well-posedness}\label{s:preliminaries}

Since the diffusion term degenerates when $u=0$ a loss of
regularity of solutions occurs. We define below a notion of weak
solution for Problem $\Pe$, which is very similar to the one
proposed by Aronson, Crandall and Peletier \cite{ACP} for the one
dimensional problem with homogeneous Dirichlet boundary
conditions. Concerning the initial data, we suppose here that $u_0
\in L^\infty (\om)$ and $u_0 \geq 0$ a.e. Note that in this
subsection, and only in this subsection, we assume, without loss
of generality, that $\ep=1$, which yields the Problem
\[
 (P) \quad\begin{cases}
 u_t=\Delta (u^m)+f(u) &\text{in }\om \times (0,\infty)\vspace{3pt}\\
 \di \frac{\partial (u^m)}{\partial \nu} = 0 &\text{on }\partial \om \times (0,\infty)\vspace{3pt}\\
 u(x,0)=u_0(x) &\text{in }\om\,.
 \end{cases}
\]

\begin{defi}\label{definition-weaksol-poreux}
A function $u:[0,\infty)\to L^1(\om)$ is a solution of Problem
$(P)$ if, for all $T>0$,
\begin{enumerate}
\item $u \in C\left([0,\infty);L^1(\om)\right)\cap L^\infty
(\Q)\,;$
  \item for all $\varphi \in C^2(\overline{\Q})$ such that $\varphi \geq
  0$ and $\di \frac {\partial \varphi}{\partial \nu}=0$ on $\partial
  \om$, it holds that
\begin{equation}\label{deqexi-poreux}
\int_ \om u(T)\varphi(T)-\int\int_{\Q}(u\varphi _t+u^m\Delta
\varphi)=\int _\om u_0\varphi(0)+\int\int _{\Q} f(u)\varphi\,.
\end{equation}
\end{enumerate}
A sub-solution (respectively a super-solution) of Problem $(P)$ is
a function satisfying (i) and (ii) with equality replaced by
$\leq$ (respectively $\geq$).
\end{defi}

\begin{thm}[Existence and
comparison principle]\label{Existence-comparison}The following
properties hold.
\begin{enumerate}
 \item Let $u^-$ and $u^+$ be a sub-solution and a super-solution
of Problem $(P^\varepsilon)$ with initial data $u_0^-$ and $u_0^+$
respectively.
$$
\mbox{If~}\quad u_0^- \leq u_0^+\,\quad{a.e.}\quad
\mbox{~then~}\quad u^-\leq u^+ \mbox{~in~} Q_T\,;
$$
 \item Problem $(P)$ has a unique solution $u$ on
$[0,\infty)$ and
\begin{equation}\label{encadrement}
0\leq u \leq \max(1,\Vert u_0 \Vert
_{L^\infty(\om)})\quad\mbox{~in~} Q_T\,;
\end{equation}
\item $u\in C(\overline{Q_T})$.
\end{enumerate}
\end{thm}

The proof of the theorem above can be performed in the same lines
as in \cite[Theorem 5]{ACP} (see also \cite{Kal} and \cite{BKP}
 for related results). The continuity of $u^\varepsilon$
follows from \cite{DB}.

The following result turns out to be an essential tool when
constructing smooth sub- and super-solutions of Problem $\Pe$.

\begin{lem}\label{lemma-sup-poreux}
Let $u$ be a continuous nonnegative function in $\ombar \times
[0,T]$. Define $\support=\{x\in\om:\, u(x,t)>0\}$ and
$\gammasupport=\partial \support$ for all $t\in[0,T]$. Suppose the
family $\Gamma:=\cup _{0< t \leq T} \gammasupport  \times \{t\}$
is sufficiently smooth and let $\nusupport$ be the outward normal
vector on $\gammasupport$. Suppose moreover that
\begin{enumerate}
\item $\n (u^m)  \text{ is continuous in }\ \ombar \times [0,T]$
\item ${\cal L}[u]:=u_t-\Delta (u^m)-f(u)=0\; \text{ in }\
\{(x,t)\in \ombar \times [0,T]:\, u(x,t)>0\}$ \item
$\di{\frac{\partial (u^m)}{\partial \nusupport}}=0\; \text{ on }\
\partial \support,\; \text{ for all }\ t\in[0,T]$\,.
\end{enumerate}
Then $u$ is a solution of Problem $(P)$. Similarly a sub-solution
(respectively a super-solution) of Problem $(P)$ is a function
satisfying (i) and (ii)---(iii) with equality replaced by $\leq$
(respectively $\geq$).
\end{lem}

The proof of this result can be found in \cite{HKLM}.

\section{Motion of the transition layers}\label{s:motion}

\subsection{Materials}\label{ss:materials}

\noindent{\bf Underlying travelling waves.} Hosono \cite{Hos} has
investigated travelling wave solutions for the degenerate one
dimensional equation
$$u_t=(u^m)_{xx}+f(u)\,.$$
 He proved that there exists a unique
travelling wave $(c^*,U)$, that the sign of the velocity $c^*$ is
that of $\int _ 0 ^1 u^{m-1}f(u)\,du$, and that the profiles vary
with the sign of the velocity. More precisely, for $c^*<0$, the
front is smooth and $U\in C^\infty(\R)$, whereas, for $c^*>0$, we
only have $({U}^{m-1})'\in L^\infty (\R)$, but $({U}^{m-1})'\notin
C(\R)$. These different behaviors of the travelling waves are in
contrast with the density independent diffusion models, where
fronts are smooth whatever their velocities are (see \cite{FM}).

In the present paper, the assumption \eqref{positive-speed}
implies that $c^*>0$. More precisely the following holds (see
\cite{Hos} for details). The travelling wave $(c^*,U)$ is the
solution of the auxiliary problem
\begin{equation}\label{eq-phi-poreux-probleme}
\left\{\begin{array}{ll}
({U}^m) '' +c^*U'+f(U)=0 &\text{ on } (-\infty,\omega) \vspace{5pt}\\
U(-\infty)= 1\vspace{5pt}\\
 U(0)=a\vspace{5pt}\\
 {U}'<0 &\text{ on } (-\infty,\omega)\vspace{5pt}\\
({U}^m)'(\omega)=0\vspace{5pt}\\
 U\equiv 0 &\text{ on
}[\omega,\infty)\,,
\end{array} \right.
\end{equation}
for some $\omega >0$. As $z\to -\infty$, terms are exponentially
decaying:
\begin{equation}\label{est-phi-U}
\max\left(1-U(z),|{U}'(z)|,|{U}''(z)|\right) \leq Ce^{-\lambda
|z|}\quad \text{ for } z\leq 0\,,
\end{equation}
for some positive constants $C$ and $\lambda$. As
$z\nearrow\omega$, we have
\begin{equation}\label{derivee-en-omega} \lim _{z
\nearrow\omega} \left({U}^{m-1}\right)'(z)=-\frac{m-1}m c^*\quad
\text{ and }\quad \lim _{z\nearrow
\omega}({U}^{m-1})''(z)=-\frac{(m-1)^2}{m^2}f'(0)\,,
\end{equation}
and $U'(\omega)\in[-\infty,0)$. Moreover, for a positive constant
which we denote again by $C$, there holds
\begin{equation}\label{estimate}
|(U^m)''(z)|\leq C |U'(z)|=-CU'(z)\quad\text{ for all }
z\in(-\infty,\omega)\,.
\end{equation}

\vskip 8pt \noindent{\bf The cut-off signed distance function.}
Another classical ingredient in similar situations (see \cite{Ser}
or \cite{Gil-Tru}) is a {\it cut-off signed distance function $d$}
which we now define. Let $\widetilde d(\cdot,t)$ be the signed
distance function to $\Gamma _t$, namely
\begin{equation}\label{eq:dist}
\widetilde d (x,t):=
\begin{cases}
-&\hspace{-10pt}{\rm dist}(x,\Gamma _t)\quad\text{ for }x\in\regionun \\
&\hspace{-10pt} {\rm dist}(x,\Gamma _t) \quad \text{ for }
x\in\regionzero\,,
\end{cases}
\end{equation}
where ${\rm dist}(x,\Gamma _t)$ is the distance from $x$ to the
hypersurface  $\Gamma  _t$. We remark that $\widetilde
d(\cdot,t)=0$ on $\Gamma _t$ and that $|\nabla \widetilde d|=1$ in
a neighborhood of the interface, say $|\n \widetilde d (x,t)|=1$
if $|\widetilde d (x,t)|<2d_0$, for some $d_0 >0$. By reducing
$d_0$ if necessary we can assume that $\widetilde d$ is smooth in
$ \{(x,t) \in \ombar \times [0,T] :\,|\widetilde{d}(x,t)|<3
 d_0\}$ and that
\begin{equation}\label{front}
 dist(\Gamma_t,\partial \Omega)\geq 3d_0 \quad \textrm{ for all }
 t\in[0,T]\,.
\end{equation}

 Next, let $\zeta(s)$ be a smooth increasing function on $\R$ such
that
\[
 \zeta(s)= \left\{\begin{array}{ll}
 s &\textrm{ if }\ |s| \leq d_0\vspace{4pt}\\
 -2d_0 &\textrm{ if } \ s \leq -2d_0\vspace{4pt}\\
 2d_0 &\textrm{ if } \ s \geq 2d_0\,.
 \end{array}\right.
\]
We then define the cut-off signed distance function $d$ by
\begin{equation}
d(x,t):=\zeta\left(\tilde{d}(x,t)\right)\,.
\end{equation}

Note that
\begin{equation}\label{norme-un}
\text{ if } \quad |d(x,t)|< d_0 \quad \text{ then }\quad |\nabla
d(x,t)|=1\,,
\end{equation}
that $d$ is constant ($=2d_0$) in a neighborhood of $\partial
\Omega$, and that the equation of motion $\Pz$ yields
\begin{equation}\label{interface}
 \text{ if } \quad |d(x,t)|< d_0 \quad \text{ then }\quad  d_t (x,t)+c^*=0\,.
\end{equation}
Moreover, there exists a constant $C>0$ such that
\begin{equation}\label{est-dist}
|\nabla d (x,t)|+|\Delta d (x,t)|\leq C\quad \textrm{ for all }
(x,t) \in \overline {\Q}\,.
\end{equation}

\subsection{Construction of sub-solutions}\label{ss:sub}

Equipped with the travelling wave $(c^*,U)$ and the signed
distance function $d$, we are looking for sub-solutions in the
form
\begin{equation}\label{sub-sol}
u^- _\ep (x,t):=(1-\ep)U\left(\frac{d(x,t)+\ep|\ln \ep|p\,
e^{t}}\ep\right)=(1-\ep)U(z^-_\ep(x,t))\,,
\end{equation}
where
\begin{equation}\label{def:z-moins}
z^- _\ep (x,t):=\di \frac{d(x,t)+\ep|\ln \ep|p\, e^{t}}\ep\,.
\end{equation}

\begin{lem}[Sub-solutions]\label{lem:sub}
Let $p>0$ be arbitrary. Then, for $\ep>0$ small enough, $u ^-
_\ep$ is a sub-solution for Problem $\Pe$.
\end{lem}

\begin{proof} In this proof (and only in this proof) we set $u^-_\ep=u$ and
$z^-_\ep=z$. Note that
\begin{equation}\label{machin}
\support=\{x\in \Omega:\, d(x,t)<-\ep|\ln \ep|p\,
e^{t}+\ep\omega\}\,,
\end{equation}
 where
$\support$ is defined as in Lemma \ref{lemma-sup-poreux}. It
follows that $u\equiv 0$ near the boundary $\partial \Omega$ so
that the Neumann boundary condition $(iii)$ in Lemma
\ref{lemma-sup-poreux} is fulfilled. Since $(U^m)'(\omega)=0$ we
see that $\nabla (u^m)$ is continuous in $\ombar \times [0,T]$.
Therefore, by virtue of Lemma \ref{lemma-sup-poreux}, it is enough
to prove that
$$
\ep \mathcal L ^\ep [u]:=\ep  u _t-\ep ^2 \Delta (u^m)- f(u)\leq 0
$$
in $\{(x,t):\, d(x,t)<-\ep|\ln \ep|p e^{t}+\ep\omega\}=\{(x,t):\,
z(x,t)<\omega\}$.

By using straightforward computations we get
$$
\begin{array}{lll}
\ep  u_t=(1-\ep)\left(d_t+\ep|\ln\ep|p\, e^{ t}\right)U'(z)\vsp\\
\ep ^2\Delta (u^m) = (1-\ep)^m|\n d|^2(U^m)''(z)+(1-\ep) ^m \ep
\Delta d (U^m)'(z)\,,
\end{array}
$$
where $z=z(x,t)$. Then using the ordinary differential equation
$(U^m)''+c^*U'+f(U)=0$, we see that
$$
\ep \mathcal L^\ep [u]=E_1+\cdots+E_4\,,
$$
with
\vsp \\
$\qquad\quad  E_1:=(1-\ep)\left[d_t+c^*-(1-\ep)^{m-1}\ep \Delta d (m U^{m-1})(z)+\ep |\ln \ep|p\, e^{ t}\right]U'(z)$\vsp \\
$\qquad\quad  E_2:=(1-\ep)^m (1-|\nabla d|^2)(U^m)''(z)$\vsp \\
$\qquad\quad  E_3:=\left((1-\ep)-(1-\ep)^m\right)(U^m)''(z)$\vsp \\
$\qquad\quad  E_4:=-f\left((1-\ep)U(z)\right)+(1-\ep)f(U(z))\,.$\vsp\\
In the following we shall denote by $C$ some positive constants
which do not depend on $\ep>0$ small enough (and may change from
place to place).

We start with some observations on the term $E_4$. Note that
\begin{equation}\label{truc}
f\left((1-\ep)u\right)-(1-\ep)f(u)=-\ep u f'(\theta)+\ep f(u)\,,
\end{equation} for some $\theta \in ((1-\ep)u,u)$. Hence
$$
|E_4|\leq C\ep\,.
$$
Moreover, since $f(1)=0$ and $f'(1)<0$, it follows from
\eqref{truc} that, for $u$ sufficiently close to 1,
\begin{equation}\label{prop}
 f\left((1-\ep)u\right)-(1-\ep)f(u)\geq \beta \ep u\,,
\end{equation}
for some $\beta >0$. Hence since $U(-\infty)=1$, by choosing
$\gamma \gg 1$ we see that
\begin{equation}\label{bidule} E_4\leq
-\beta \ep U(z) \leq -\frac 1 2 \beta \ep\quad\text{ for all
}z\leq -\gamma\,.
\end{equation}
In the following we distinguish three cases, namely \eqref{case1},
\eqref{case2} and \eqref{case3}.

\vsp Assume that
\begin{equation}\label{case1}
-\ep |\ln \ep| p e^{t} - \ep \gamma\leq d(x,t)< -\ep |\ln \ep|  p
e^{t} + \ep \omega\,,
\end{equation}
which in turn implies that $-\gamma\leq z<\omega$. Since $U'<0$ on
$(-\infty,\omega)$ and $U'(\omega)\in[-\infty,0)$, it holds that
$U'(z)\leq -\alpha$, for some $\alpha
>0$. If $\ep
>0$ is small enough \eqref{norme-un} shows that $E_2=0$; from \eqref{estimate}
we deduce that $|E_3|\leq -C\ep U'(z)$; moreover we have
$|E_4|\leq C\ep$. In view of \eqref{interface}, $E_1$ reduces to
\begin{equation}\label{reduces}
E_1= (1-\ep)\left[-(1-\ep)^{m-1}\ep \Delta d (m U^{m-1})(z)+\ep
|\ln \ep|p\, e^{ t}\right]U'(z)\,. \end{equation} Since
$|-(1-\ep)^{m-1}\ep \Delta d (m U^{m-1})(z)|\leq C\ep$, inequality
$E_1 \leq \frac 12 p\, \ep |\ln \ep| U'(z)$ holds. Collecting
theses estimates we have $$ \begin{array}{ll}{\mathcal L }^\ep
[u]&\leq
(\frac 12 p\, \ep |\ln \ep|-C\ep)U'(z) +C\ep\vsp\\
&\leq -\frac 1 4 p\, \alpha  \ep |\ln \ep|+C\ep\vsp\\
&\leq 0\,,
\end{array}
$$
if $\ep >0$ is sufficiently small.

\vsp Assume that
\begin{equation}\label{case2}
-d_0 \leq d(x,t)< -\ep |\ln \ep| p e^{ t} -\ep \gamma\,,
\end{equation}
which in turn implies that $ z<-\gamma$ so that \eqref{bidule}
implies $E_4 \leq 0$ . Here again \eqref{norme-un} shows that
$E_2=0$, from \eqref{estimate} we deduce that $|E_3|\leq -C\ep
U'(z)$, and $E_1$ reduces to \eqref{reduces}. Hence we collect
$$
{\mathcal L}^\ep [u]\leq U'(z)\left[-(1-\ep)\ep C+ (1-\ep) p\, \ep
|\ln \ep|-C\ep\right]\leq 0\,,
$$
for $\ep >0$ small enough.

 \vsp Assume that
\begin{equation}\label{case3}
-2d_0 \leq d(x,t)< -d_0\,,
\end{equation}
which in turn implies that, for $\ep >0$ small enough, $ z\leq
-\frac {d_0}{2\ep}$. In this range \eqref{norme-un} and
\eqref{interface} no longer apply but the exponential decay
\eqref{est-phi-U} shows that $|E_1|+|E_2|+|E_3|\leq C e^{-\lambda
\frac{d_0}{2\ep}}$.  Last $E_4 \leq - \frac 12\beta \ep$ (see
\eqref{bidule}) shows that, for $\ep >0$ small enough, ${\mathcal
L}^\ep[u]\leq 0$.

The lemma is proved.
\end{proof}

\subsection{Construction of super-solutions}\label{ss:super}

The construction of super-solutions is more involved: since we
want them to be positive it is no longer possible to use the
natural travelling wave $(c^*,U)$ which is compactly supported.
Therefore we shall first consider slightly larger speeds $c>c^*$
which provide faster travelling wave solutions which tend to
$+\infty$ in $-\infty$; then a small modification will provide us
positive and more regular functions which are \lq\lq nearly"
travelling wave solutions. Before making this argument more
precise, let us note that, as it will clearly appear below, the
possibility of the above strategy follows from \cite{Hos}.

\vskip 8 pt Let $\eta \in(0,\min(a,1-a))$ be arbitrary. Let $\epai
>0$ be fixed. Let us recall that we have fixed $0<T<T^{max}$, where
 $T^{max}$ denotes the time when the first
singularities occur in $\Pz$. Therefore we can select $\rho >0$
small enough so that the following holds. First the smooth
solution $(\Gamma _t ^c)$ of the free boundary problem
\[
 (P^0 _c)\quad\begin{cases}
 \, V_{n}=c:=c^*+\rho
 \quad \text { on } \Gamma_t ^c \vspace{3pt}\\
 \, \Gamma_t ^c \big|_{t=0}=\Gamma_0\,,
\end{cases}
\]
exists at least on $[0,T]$. Secondly, if we denote by $d^c(x,t)$
the cut-off signed distance function associated with $\Gamma
^c:=\cup _{0\leq t \leq T} (\Gamma_t ^c\times\{t\})$ then, for all
$(x,t) \in \Q$,
\begin{equation}\label{d-rho}
d(x,t)\geq \epai \Longrightarrow d^c(x,t)\geq \frac \epai 2\,.
\end{equation}

Since $c>c^*$, as explained in \cite[Remark 3.1]{Hos}, there
exists a faster travelling wave $(c,V)$ which satisfies the same
requirements as $(c^*,U)$ in the auxiliary problem
\eqref{eq-phi-poreux-probleme}, except that $V(-\infty)=+\infty$
rather than $U(-\infty)=1$. In particular, $V$ is still compactly
supported from one side.

 Next, for all $n\geq 1$, following the construction which comes before
Proposition 4.1 in \cite{Hos} (it consists in slightly modifying
the above travelling wave $(c,V)$), we can consider $(c,U_n)$ such
that
\begin{enumerate}
\item $U_n$ satisfies the ordinary differential equation
$$({U_n}^m)''+c{U_n}'+f(U_n)=0\quad\text{ on some $(-\infty,
Z_n)\,,$}$$ where ${U_n}'<0$ holds
 \item $U_n$ is
constant equal to some $(\delta _n)^{\frac 1{m-1}}>0$ on
$[Z_n,\infty)$ \item $(U_n)^{m-1} \in C^1(\R)$
\end{enumerate}
together with $U_n(0)=a$ and $U_n(-\infty)=+\infty$. Moreover
$\delta _n \to 0$ as $n\to \infty$, so that we can fix $n_0 \gg1$
such that $(\delta _{n_0}) ^{\frac 1{m-1}}\leq \eta$.

As a conclusion, if we denote $U_{n_0}$, $\delta _{n_0}$ and
$Z_{n_0}$ by $W$, $\delta$ and $Z$ we are now equipped with
$(c,W)$ such that $W^{m-1} \in C^1(\R)$ and
\begin{equation}\label{doublev}
\left\{\begin{array}{ll}
({W}^m) '' +c W'+f(W)=0 &\text{ on } (-\infty,Z) \vspace{5pt}\\
W(-\infty)= +\infty\vspace{5pt}\\
 W(0)=a\vspace{5pt}\\
 {W}'<0 &\text{ on } (-\infty,Z)\vspace{5pt}\\
 W\equiv \delta  ^{\frac
1{m-1}}\leq \eta &\text{ on }[Z,\infty)\,.
\end{array} \right.
\end{equation}

\vskip 8 pt We are now looking for super-solutions in the form
\begin{equation}\label{super-sol}
u_\ep ^+(x,t):=W\left(\frac{d^c(x,t)-\ep |\ln
\ep|Ke^t}{\ep}\right)\,.
\end{equation}
In the sequel we set
\begin{equation}\label{def:z-plus}
z^+ _\ep (x,t):=\di \frac{d^c (x,t)- \ep|\ln \ep|Ke^t}\ep\,.
\end{equation}

\begin{rem}[The sub-domain $\Sigma$]
We shall consider below a sub-domain $\Sigma$ whose slice at time
$t$, namely $\sigma_t:=\{x:\,(x,t)\in \Sigma\}$, is the region
enclosed between $\partial \Omega$ and (more or less) $\Gamma ^c
_t$. We shall prove that $u_\ep ^+$ is a super-solution in
$\Sigma$. Thanks to \eqref{avantage} this will be sufficient for
our purpose (see Section \ref{s:proof}).
\end{rem}

Denote by $-\theta$ the point where $W(-\theta)=1+\eta$. For each
$0\leq t\leq T$, define the open set
$$
\sigma _t:=\{x\in\Omega:\, d^c(x,t) > \ep |\ln \ep|K e^t-\ep
\theta\}=\{x:\,z^+_\ep (x,t) >-\theta\}\,,
$$
and the sub-domain
$$
\Sigma:=\cup _{0 < t < T} (\sigma _t \times \{t\})\,.
$$
Note that the lateral boundary of $\Sigma$ is made of $\partial
_{\text{out}} \Sigma:=\partial \Omega \times (0,T)$ and $\partial
_{\text{in}} \Sigma:=\cup _{0<t<T} (s_t \times \{t\})$ where $s_t$
denotes the smooth hypersurface
$$
s_t:=\{x\in\Omega:\, d^c(x,t) = \ep |\ln \ep|K e^t-\ep
\theta\}=\{x:\,z^+_\ep (x,t) =-\theta\}\,.
$$

\begin{lem}[Super-solutions in $\Sigma$]\label{lem:sup}Let $\eta \in(0,\min(a,1-a))$ be arbitrary and let $\epai
>0$ be fixed. Then, for all $K>0$, all $\ep>0$ small enough,
$u ^+ _\ep$ is such that
\begin{enumerate}
\item ${\cal L}^\ep[u^+ _\ep]:=(u^+_\ep)_t-\ep \Delta
((u^+_\ep)^m)-\eun f(u^+_\ep)\geq 0\; \text{ in } \Sigma$
 \item
$\di{\frac{\partial ((u^+_\ep)^m)}{\partial \nu}}=0\; \text{ on }\
\partial
_{\text{out}} \Sigma=\partial \Omega \times (0,T)$\, \item
$u^+_\ep \equiv 1+\eta\;\text{ on } \partial _{\text{in}}
\Sigma=\cup _{0<t<T} (s_t \times \{t\})\,.$
\end{enumerate}
\end{lem}

\begin{proof} In this proof (and only in this proof) we put $u^+_\ep=u$ and
$z^+_\ep=z$. Recall that $d^c$ is constant near the boundary
$\partial \Omega$ so that the Neumann boundary condition $(ii)$ is
fulfilled. Moreover the Dirichlet boundary condition $(iii)$ is
clear from the definition of $s_t$ and the fact that
$W(-\theta)=1+\eta$.

Therefore it remains to prove that $ \ep \mathcal L ^\ep [u]=\ep u
_t-\ep ^2 \Delta (u^m)- f(u)\geq 0$ in $\Sigma=\{(x,t):\;
z(x,t)>-\theta\}$. If $z(x,t)\geq Z$ then $\mathcal L ^ \ep
[u]=\mathcal L ^\ep [\delta^{\frac 1{m-1}}] \geq 0$. We now assume
that $z(x,t)\in(-\theta,Z)$, i.e.
\begin{equation}\label{distance-petite}
\ep|\ln \ep|Ke^t-\ep\theta  <d^c (x,t)<\ep|\ln \ep|Ke^t+ \ep Z\,.
\end{equation}
Straightforward computations combined with the ordinary
differential equation $(W^m)''+cW'+f(W)=0$ yield
$$
\ep \mathcal L^\ep [u]=(d^c_t+c)W' -\ep|\ln\ep|Ke^t W'-\ep \Delta
d^c (W^m)'+(1-|\nabla d^c|^2)(W^m)''\,.
$$
If $\ep>0$ is small enough, then \eqref{distance-petite} combined
with \eqref{norme-un} and \eqref{interface} --- with $d^c$ and $c$
playing the roles of $d$ and $c^*$--- shows that the above
equality reduces to
$$\ep \mathcal L^\ep [u]=-\ep
W' \left(|\ln\ep|K e^t+\Delta d^c (m W^{m-1})\right)\,.$$ Since
$W'\leq 0$ we have $\ep \mathcal L^\ep [u]\geq 0$, for $\ep
>0$ small enough.

The lemma is proved.
\end{proof}

\section{Proof of Theorem \ref{width}}\label{s:proof}

\subsection{A generation of interface property}

We first state a result on the generation of interface.

\begin{lem}[Generation of interface]\label{lem:generation} Assume
$m\geq 2$. Let $\eta>0$ be arbitrary small.  Then, for all $x\in
\Omega$, we have, for $\ep>0$ small enough,
\begin{equation}\label{machin1}
0\leq \ue(x,t^\ep)\leq 1+\eta\,,
\end{equation}
and there exists $M_0>0$ such that, for $\ep>0$ small enough,
\begin{align}
&u_0(x) \geq a+ M_0 \ep|\ln \ep| \Longrightarrow \ue(x,t^\ep) \geq 1-\ep\label{machin2}\\
&u_0(x) \leq a- M_0 \ep|\ln \ep| \Longrightarrow \ue(x,t^\ep) \leq
\ep\,,\label{machin3}
\end{align}
where $t^\ep=\mu^{-1}\ep |\ln\ep|$.
\end{lem}

\begin{proof} We only give an outline since the arguments can be found in \cite{A-Hil-Mat}, \cite{A-Hil}.

Denote by $Y(\tau;\xi)$ the solution of the bistable ordinary
differential equation $Y_\tau=f(Y)$ on $(0,\infty)$ supplemented
with the initial condition $Y(0;\xi)=\xi$. Modulo a change of the
time variable, we can use the sub- and super-solutions constructed
in \cite{A-Hil} to deduce that, for some $C^*>0$, for $\ep >0$
small enough,
\begin{multline*}
\left[Y\left(\frac{t^\ep}{\ep};\,u_0(x)-\ep^2
C^*(\emutep-1)\right)\right]^+ \\
\leq \ue(x,t^\ep)\leq \left[Y\left(\frac{t^\ep}{\ep};\,u_0(x)
+\ep^2 C^*(\emutep-1)\right)\right]^+\,.
\end{multline*}

 Next a straightforward modification of \cite[Lemma
3.9]{A-Hil-Mat} ---which specifies the instability of the
equilibrium $Y\equiv a$ and the stability of the equilibria
$Y\equiv 0$, $Y\equiv 1$--- shows that for $\ep
>0$ small enough, for all $\xi\in (-1,2)$, we have
$$
Y(\mu ^{-1} | \ln \ep |;\xi) \leq 1+\eta\,,
$$
and that
\begin{align*}
&\xi \geq a+  \ep|\ln \ep| \Longrightarrow Y(\mu ^{-1}| \ln \ep
|;\xi) \geq 1-\ep\\
&\xi \leq a-  \ep|\ln \ep| \Longrightarrow Y(\mu ^{-1}| \ln \ep
|;\xi) \leq \ep\,.
\end{align*}

The combination of the above arguments completes the proof of the
lemma.
\end{proof}

\subsection{Proof of the main theorem}

We are now in the position to prove Theorem \ref{width}. To that
purpose we show that solutions are in between the propagation of
interface sub- and super-solutions at time $t^\ep$.

\begin{proof} Assume $m\geq 2$. Let $\eta \in(0,\min(a,1-a))$ be arbitrary.
First note that the comparison principle directly implies that
inequality \eqref{machin1} persists for later times, i.e.
\begin{equation}\label{avantage}
\ue(x,t)\in[0,1+\eta]\quad\text{ for all $x\in \Omega$ and all
$t^\ep \leq t\leq T\,.$}
\end{equation}

Since $\n u _0 \neq 0$ everywhere on $\Gamma _0=\{x\in\om:\,
u_0(x)=a\}$ and since $\Gamma _0$ is a compact hypersurface, we
can find a positive constant $M_1$ such that
\begin{equation}\label{corres}
\begin{array}{ll}\text { if } \quad d(x,0)  \leq \ -M_1 \ep
|\ln \ep|
&\text { then } \quad u_0(x) \geq a +M_0 \ep|\ln\ep| \vspace{3pt}\\
\text { if } \quad d(x,0)  \geq {}M_1 \ep |\ln\ep|& \text { then }
\quad u_0(x) \leq a -M_0 \ep |\ln\ep|\,,
\end{array}
\end{equation}
with $M_0$ the constant which appears in Lemma
\ref{lem:generation}.

We first investigate the behavior \lq\lq inside the interface". In
view of \eqref{machin}, we can choose $p>0$ large enough so that,
for $\ep>0$ small enough, the sub-solution $u _\ep ^-$ defined in
\eqref{sub-sol} is such that the set $\{x:\, u_\ep ^-(x,0)>0\}$ is
included in $\{x:\; d(x,0)\leq -M_1 \ep |\ln \ep|\}$. Therefore,
from the correspondence \eqref{corres}, the estimate
\eqref{machin2} and the fact that $u _\ep ^- (x,0)\leq 1- \ep$, we
deduce that, for all $x\in \Omega$,
$$
u_ \ep ^- (x,0)\leq \ue(x,t^\ep)\,.
$$
It follows from the comparison principle that
\begin{equation}\label{dessous}
u_\ep^-(x,t-t^\ep)=(1-\ep)U\left(\frac{d(x,t-t^\ep)+\ep|\ln
\ep|p\, e^{t-t^\ep}}\ep\right) \leq u^\ep (x,t)\,,
\end{equation}
for all $(x,t)\in\om\times [t^\ep,T]$. We choose $\mathcal C\gg
\max(c^* \mu^{-1},p\,e^T,\lambda)$ so that
\begin{equation}\label{choice} \frac{-\mathcal C
\ep|\ln\ep|+c^*\mu^{-1}\ep|\ln\ep|+\ep|\ln \ep|p\,
e^{t-t^\ep}}\ep\leq -\frac {\mathcal C} 2 |\ln \ep|\leq -\frac 2
\lambda |\ln \ep|\,,
\end{equation}
where $\lambda>0$ is the constant appearing in \eqref{est-phi-U}.
We take $x\in \om _t^{(1)}\setminus\mathcal
N_{\C\ep|\ln\ep|}(\Gamma _t)$, i.e.
\begin{equation}\label{d-moins}
d(x,t)\leq -\mathcal C \ep|\ln \ep|\,,
\end{equation}
 and prove below that $\ue(x,t)\geq 1-2\ep $, for $t^\ep\leq t\leq
 T$. Note that, for $\ep >0$ small enough,
\begin{equation}\label{lien}
 d(x,t-t^\ep)=d(x,t)+c^* t^\ep\,,
\end{equation}
which, combined with \eqref{dessous} and \eqref{choice}, implies
$$
\begin{array}{lll}
 \ue(x,t) &\geq
(1-\ep)U\left(-\frac 2 \lambda|\ln\ep|\right)\vsp\\
&\geq (1-\ep)(1-C\ep ^2)\vsp\\
 &\geq 1-2\ep\,,
\end{array}
$$
where we have used \eqref{est-phi-U}.

Last we investigate the behavior \lq\lq outside the interface".
Fix $\epai>0$ arbitrarily small. For such $\epai >0$, we follow
the strategy of subsection \ref{ss:super} to construct
super-solutions in $\Sigma$. More precisely define $K=2M_1$, where
$M_1>0$ is the constant that appears in \eqref{corres}; then
construct $u_\ep ^+(x,t)$ as in \eqref{super-sol}. In view of
Lemma \ref{lem:sup} $(ii)$, $(iii)$, the super-solution
$u^+_\ep(x,t)$ and the solution $\ue(x,t+t^\ep)$ satisfy a Neumann
boundary condition on $\partial _{\text{out}} \Sigma$ and  ---
taking advantage of \eqref{avantage}--- are ordered on $\partial
_{\text{in}}\Sigma$. Last we claim that (note that
$d^c(x,0)=d(x,0)$ since $\Gamma _0 ^c=\Gamma _0$)
\begin{equation}\label{ordre}
 \ue(x,t^\ep)\leq W\left(\frac{d(x,0)-2M_1\ep |\ln
\ep|}{\ep}\right)=u_ \ep ^+ (x,0)\,,
\end{equation}
for all  $x\in \sigma _0 =\{x:\;d(x,0)>2M_1\ep|\ln \ep|-\ep
\theta\}$. Indeed, \eqref{corres} and \eqref{machin3} show that
$\ue(x,t^\ep)\leq \ep$ and the conclusion \eqref{ordre} follows
from the fact that $\delta ^{\frac 1 {m-1}}\leq W$. Hence the
comparison principle yields
\begin{equation}\label{dessus} u^\ep
(x,t) \leq W\left(\frac{d^c(x,t-t^\ep)- \ep |\ln \ep|2 M_1
e^{t-t^\ep}}{\ep}\right)= u_\ep^+(x,t-t^\ep) \,,
\end{equation}
for all $(x,t)\in \Sigma$ with $t^\ep \leq t \leq T$. We take
$x\in \om _t^{(0)}\setminus\mathcal N_{\epai}(\Gamma _t)$, i.e.
\begin{equation}\label{d-plus}
d(x,t)\geq \epai\,,
\end{equation}
 and prove below that $\ue(x,t)\leq \eta $, for $t^\ep\leq t\leq
 T$. From \eqref{d-rho} we deduce that $(x,t)\in \Sigma$ so that \eqref{dessus} applies.
 Note also that $ d^c(x,t-t^\ep)=d^c(x,t)+c t^\ep$. Therefore we
 infer from \eqref{d-rho} that, for $\ep >0$ small enough,
 $d^c(x,t-t^\ep)\geq\frac \epai 3$ which in turn implies
 $$
 u_\ep ^+(x,t-t^\ep)\leq W\left(\frac{\frac \epai 3 -2 M_1 e^T \ep
|\ln \ep|}{\ep}\right)= \delta  ^{\frac 1{m-1}}\leq \eta\,,
$$
since $W(+\infty)= \delta  ^{\frac 1{m-1}}$. Conclusion follows
from \eqref{dessus}.

 Theorem
\ref{width} is proved.\end{proof}

\end{document}